\documentclass[11pt,a4paper]{amsart}
\usepackage{amsmath,amsthm,amsfonts,graphicx}
\usepackage[colorlinks=true,citecolor=black,linkcolor=black]{hyperref}
\newcommand{\doi}[1]{\href{http://dx.doi.org/#1}{\texttt{doi:#1}}}
\usepackage[numbers,sort&compress]{natbib}
\usepackage[lmargin=40mm,rmargin=40mm,tmargin=34.25mm,bmargin=34.25mm]{geometry}
\renewcommand{\baselinestretch}{1.25}
\theoremstyle{plain}
\newtheorem{theorem}{Theorem}[section]
\newtheorem{lemma}[theorem]{Lemma}

\theoremstyle{definition}

\newcommand{\thmlabel}[1]{\label{thm:#1}}
\newcommand{\thmref}[1]{Theorem~\ref{thm:#1}}
\newcommand{\twothmref}[2]{Theorems~\ref{thm:#1} and \ref{thm:#2}}
\newcommand{\lemlabel}[1]{\label{lem:#1}}
\newcommand{\lemref}[1]{Lemma~\ref{lem:#1}}
\newcommand{\seclabel}[1]{\label{sec:#1}}
\newcommand{\secref}[1]{Section~\ref{sec:#1}}
\newcommand{\figlabel}[1]{\label{fig:#1}}
\newcommand{\figref}[1]{Figure~\ref{fig:#1}}

\newcommand{\CP}[2]{\ensuremath{#1{\,\square\,}#2}}
\newcommand{\X}{\ensuremath{\mathcal{X}}}
\newcommand{\A}{\ensuremath{\mathcal{A}}}
\newcommand{\B}{\ensuremath{\mathcal{B}}}
\newcommand{\PP}{\ensuremath{\mathcal{P}}}
\newcommand{\Q}{\ensuremath{\mathcal{Q}}}
\newcommand{\C}[1]{\ensuremath{P\langle#1\rangle}}
\newcommand{\BB}[1]{\ensuremath{\B\langle#1\rangle}}
\newcommand{\ceil}[1]{\ensuremath{\left\lceil#1\right\rceil}}

\begin{document}

\title[Grid-like-minors]{Polynomial treewidth forces\\ a large grid-like-minor}

\author{Bruce~A.~Reed}
\address{\newline Canada Research Chair in Graph Theory
\newline School of Computer Science
\newline McGill University
\newline Montr\'eal, Canada
\bigskip
\newline Laboratoire I3S
\newline Centre National de la Recherche Scientifique
\newline Sophia-Antipolis, France}
\email{breed@cs.mcgill.ca}

\author{David~R.~Wood}
\thanks{D.W.\ is supported by a QEII Research Fellowship from the Australian Research Council.}
\address{\newline Department of Mathematics and Statistics
\newline The University of Melbourne
\newline Melbourne, Australia}
\email{woodd@unimelb.edu.au}


\date{\today}

\begin{abstract}
Robertson and Seymour proved that every graph with sufficiently large treewidth contains a large grid minor. However, the best known bound on the treewidth that forces an $\ell\times\ell$ grid minor is exponential in $\ell$. It is unknown whether polynomial treewidth suffices. We prove a result in this direction. A \emph{grid-like-minor of order} $\ell$ in a graph $G$ is a set of paths in $G$ whose intersection graph is bipartite and contains a $K_{\ell}$-minor. For example, the rows and columns of the $\ell\times\ell$ grid are a grid-like-minor of order $\ell+1$. We prove that polynomial treewidth forces a large grid-like-minor. In particular, every graph with treewidth at least $c\ell^4\sqrt{\log\ell}$ has a grid-like-minor of order $\ell$. As an application of this result, we prove that the cartesian product \CP{G}{K_2} contains a $K_{\ell}$-minor whenever $G$ has treewidth at least $c\ell^4\sqrt{\log\ell}$.
\end{abstract}

\maketitle

\section{Introduction}
\seclabel{Intro}

A central theorem in Robertson and Seymour's theory of graph minors states
that the grid\footnote{The $\ell\times\ell$ grid is the planar graph with
  vertex set $[\ell]\times[\ell]$, where vertices $(x,y)$ and $(p,q)$ are
  adjacent whenever $|x-p|+|y-q|=1$.} to  is a canonical witness for a graph to have large treewidth, in the sense that the $\ell\times \ell$ grid has treewidth $\ell$, and every graph with sufficiently large treewidth contains an $\ell\times\ell$ grid minor \citep{RS-GraphMinorsV-JCTB86}. See \citep{RST-JCTB94,DJGT-JCTB99,Reed97} for alternative proofs. The following theorem is the best-known explicit bound. See \citep{DHK-Algo09,DemHaj-Comb08} for better bounds under additional assumptions. 

\begin{theorem}[\citet*{RST-JCTB94}]\thmlabel{GridMinor} 
Every graph with treewidth at least $20^{2\ell^5}$ contains an $\ell\times \ell$ grid minor.
\end{theorem}

\citet{RST-JCTB94} also proved that certain random graphs have treewidth proportional to  $\ell^2\log\ell$, yet do not contain an $\ell\times\ell$ grid minor. This is the best known lower bound on the function in \thmref{GridMinor}. Thus it is open whether polynomial treewidth forces a large grid minor. This question is not only of theoretic interest---for example, it has direct bearing on certain algorithmic questions \citep{DemHaj-EuJC07}. In this paper we prove that polynomial treewidth forces a large `grid-like-minor'.

A \emph{grid-like-minor of order} $\ell$ in a graph $G$ is a set \PP\
of paths in $G$, such that the intersection graph\footnote{The
  intersection graph of a set $X$, whose elements are sets, has vertex
  set $X$ where distinct vertices are adjacent whenever the corresponding
  sets have a non-empty intersection.} of \PP\ is bipartite and contains a $K_{\ell}$-minor. Observe that the intersection graph of the rows and columns of the $\ell\times\ell$ grid is the complete bipartite graph $K_{\ell,\ell}$, which contains a $K_{\ell+1}$-minor (formed by contracting a matching of $\ell-1$ edges). Hence, the $\ell\times\ell$ grid contains a grid-like-minor of order $\ell+1$. The following is our main result.

\begin{theorem}
\thmlabel{GridLikeMinor}
Every graph with treewidth at least $c\ell^4\sqrt{\log \ell}$ contains a grid-like-minor of order $\ell$, for some constant $c$. Conversely, every graph that contains a grid-like-minor of order $\ell$ has treewidth at least $\ceil{\frac{\ell}{2}}-1$.
\end{theorem}

\thmref{GridLikeMinor} proves that grid-like-minors serve as a canonical witness for a graph to have large treewidth, just like grid minors. The advantage of grid-like-minors is that a polynomial bound on treewidth suffices. The disadvantage of grid-like-minors is that they are a broader structure than grid minors (but not as broad as brambles; see \secref{Background}). 

\thmref{GridLikeMinor} has an interesting corollary concerning the cartesian product \CP{G}{K_2}. This graph consists of two copies of $G$ with an edge between corresponding vertices in the two copies. Motivated by Hadwiger's Conjecture for cartesian products, the second author \citep{Wood-ProductMinor} showed that the maximum order of a complete minor in \CP{G}{K_2} is tied to the treewidth of $G$. In particular, if $G$ has treewidth at most $\ell$, then \CP{G}{K_2} has treewidth at most $2\ell+1$ and thus contains no $K_{2\ell+3}$-minor. Conversely, if $G$ has treewidth at least $2^{4\ell^4}$, then $\CP{G}{K_2}$ contains a $K_\ell$-minor. The proof of the latter result is based on the version of \thmref{GridMinor} due to \citet*{DJGT-JCTB99}. The following theorem is a significant improvement.

\begin{theorem}
\thmlabel{CartProd}
If a graph $G$ has treewidth at least $c\ell^4\sqrt{\log \ell}$, then $\CP{G}{K_2}$ contains a $K_\ell$-minor, for some constant $c$.
\end{theorem}

\section{Background}
\seclabel{Background}

All graphs considered in this paper are undirected, simple, and finite. For undefined terminology, see \citep{Diestel00}. A graph $H$ is a \emph{minor} of a graph $G$ if a graph isomorphic to $H$ can be obtained from a subgraph of $G$ by contracting edges. A graph $G$ is \emph{$d$-degenerate} if every subgraph of $G$ has a vertex of degree at most $d$. \citet{Mader67} proved that every graph with no $K_\ell$-minor is $2^{\ell-2}$-degenerate. Let $d(\ell)$ be the minimum integer such that every graph with no $K_\ell$-minor is $d(\ell)$-degenerate. \citet{Kostochka84} and \citet{Thomason01,Thomason84} independently proved that $d(\ell)\in\Theta(\ell\sqrt{\log \ell})$.

\begin{theorem}[\citet{Kostochka84,Thomason01,Thomason84}]\thmlabel{DegreeMinor}
Every graph with no $K_\ell$-minor is $d(\ell)$-degenerate, where $d(\ell)\leq c\ell\sqrt{\log \ell}$ for some constant $c$.
\end{theorem}

Let $G$ be a graph. Two subgraphs $X$ and $Y$ of $G$ \emph{touch} if $X\cap Y\neq\emptyset$ or there is an edge of $G$ between $X$ and $Y$. A \emph{bramble} in $G$ is a set of pairwise touching connected subgraphs. The subgraphs are called \emph{bramble elements}. A set $S$ of vertices in $G$ is a \emph{hitting set} of a bramble \B\ if $S$ intersects every element of \B. The \emph{order} of \B\ is the minimum size of a hitting set. The canonical example of a bramble of order $\ell$ is the set of crosses (union of a row and column) in the $\ell\times\ell$ grid. The following `Treewidth Duality Theorem' shows the intimate relationship between treewidth and brambles.

\begin{theorem}[\citet{SeymourThomas-JCTB93}]
\thmlabel{TreewidthBramble}
A graph $G$ has treewidth at least $\ell$ if and only if $G$ contains a bramble of order at least $\ell+1$.
\end{theorem}

See \citep{BD-CPC02} for an alternative proof of \thmref{TreewidthBramble}. In light of \thmref{TreewidthBramble}, \thmref{GridMinor} says that every bramble of large order contains a large grid minor, and \thmref{GridLikeMinor} says that every bramble of polynomial order contains a large grid-like-minor.

\section{Main Proofs}

In this section we prove \twothmref{GridLikeMinor}{CartProd}. Let
$e:=2.718\ldots$ and $[n]:=\{1,2,\dots,n\}$. The following lemma is by
\citet*{BBR07}; we include the proof for completeness.

\begin{lemma}[\citet{BBR07}]
  \lemlabel{BramblePath} Let \B\ be a bramble in a graph $G$. Then $G$
  contains a path that intersects every element of \B.
\end{lemma}

\begin{proof}
  Let $P$ be a path in $G$ that (1) intersects as many elements of \B\
  as possible, and (2) is as short as possible. Let $v$ be an endpoint
  of $P$. There is a bramble element $X$ that only intersects $P$ at
  $v$, as otherwise we could delete $v$ from $P$. Suppose on the
  contrary that $P$ does not intersect some bramble element $Z$. Since
  $X$ and $Z$ touch, there is a path $Q$ starting at $v$ through $X$
  to some vertex in $Z$, and $Q\cap P=\{v\}$. Thus $P\cup Q$ is a path
  that also hits $Z$. This contradiction proves that $P$ intersects
  every element of \B.
\end{proof}

\begin{lemma}
  \lemlabel{ManyPaths} Let $G$ be a graph containing a bramble
  $\mathcal{B}$ of order at least $k\ell$ for some integers
  $k,\ell\geq1$. Then $G$ contains $\ell$ disjoint paths
  $P_1,\dots,P_\ell$, and for distinct $i,j\in[\ell]$, $G$ contains
  $k$ disjoint paths between $P_i$ and $P_j$.
\end{lemma}

\begin{proof}
  By \lemref{BramblePath}, there is a path $P=(v_1,\dots,v_n)$ in $G$
  that intersects every element of \B. For $1\leq a\leq b\leq n$, let
  \C{a,b} be the sub-path of $P$ induced by $\{v_a,\dots,v_b\}$, and
  let $\BB{a,b}$ be the
  sub-bramble $$\BB{a,b}:=\{X\in\B:X\cap\C{a,b}\neq\emptyset,\,X\cap\C{1,a-1}=\emptyset\}\enspace.$$
  If $S$ is a hitting set of $\BB{a,b}$, then $S\cup\{v_{b+1}\}$ is a
  hitting set of $\BB{a,b+1}$.  Thus the order of $\BB{a,b+1}$ is at
  most the order of $\BB{a,b}$ plus 1.  Hence for each $a\in[n]$,
  either the order of $\BB{a,n}$ is less than $k$, or for some $b\geq
  a$ the order of $\BB{a,b}$ equals $k$.  Thus there are positive
  integers $a_1<a_2<\dots<a_s\leq n$ such that for each $i\in[s]$ the
  order of $\B_i:=\BB{a_{i-1}+1,a_i}$ equals $k$ (where $a_0=0$), and
  the order of $\B_{s+1}:=\BB{a_s+1,n}$ is less than $k$.  Since
  $\B=\B_1\cup\dots\cup\B_{s+1}$, the order of \B\ is at most
  the sum of the orders of $\B_1,\dots,B_{s+1}$, which is
  strictly less than $(s+1)k$.  Since the order of \B\ is at least
  $k\ell$, we have $s\geq\ell$.  Let $P_i:=\C{a_{i-1}+1,a_i}$ for
  $i\in[\ell]$.  Thus $P_1,\dots,P_\ell$ are disjoint paths in $G$.

  Suppose that there is a set $S\subseteq V(G)$ separating some
  pair of distinct paths $P_i$ and $P_j$, where $|S|\leq k-1$. Thus
  $S$ is not a hitting set of $\B_i$, since $\B_i$ has order
  $k$. Hence some element $X\in\B_i$ does not intersect
  $S$. Similarly, some element $Y\in\B_j$ does not intersect $S$. Thus
  $S$ separates $X$ from $Y$, and hence $X$ and $Y$ do not touch. This
  contradiction proves that every set of vertices separating $P_i$ and
  $P_j$ has at least $k$ vertices. By Menger's Theorem, there are $k$
  disjoint paths between $P_i$ and $P_j$, as desired.
\end{proof}


We now prove the main result.

\begin{proof}[Proof of the first part of \thmref{GridLikeMinor}.]
  Let
  $k:=\lceil{4e\tbinom{\ell}{2}\,d(\ell)\rceil}$.
  Let $G$ be a graph with treewidth at least $c\ell^4\sqrt{\log\ell}$,
  which is at least $k\ell-1$ for an appropriate value of $c$. By
  \thmref{TreewidthBramble}, $G$ has a bramble of order at least
  $k\ell$.  By \lemref{ManyPaths}, $G$ contains $\ell$ disjoint paths
  $P_1,\dots,P_\ell$, and for distinct $i,j\in[\ell]$, $G$ contains a
  set $\mathcal{Q}_{i,j}$ of $k$ disjoint paths between $P_i$ and
  $P_j$.

  For distinct $i,j\in[\ell]$ and distinct $a,b\in[\ell]$ with
  $\{i,j\}\neq\{a,b\}$, let $H_{i,j,a,b}$ be the intersection graph of
  $\Q_{i,j}\cup\Q_{a,b}$. Since $H_{i,j,a,b}$ is bipartite, if
  $K_\ell$ is a minor of $H_{i,j,a,b}$, then $\Q_{i,j}\cup\Q_{a,b}$ is
  a grid-like-minor of order $\ell$. Now assume that $K_\ell$ is not a
  minor of $H_{i,j,a,b}$. By \thmref{DegreeMinor}, $H_{i,j,a,b}$ is
  $d(\ell)$-degenerate.

  Let $H$ be the intersection graph of $\cup\{\Q_{i,j}:1\leq
  i<j\leq\ell\}$; that is, $H$ is the union of the $H_{i,j,a,b}$. Then
  $H$ is $\binom{\ell}{2}$-colourable, where each colour class is some
  $\Q_{i,j}$.  Each colour class of $H$ has $k$ vertices, and each
  pair of colour classes in $H$ induce a $d(\ell)$-degenerate
  subgraph. By \lemref{ChooseIndependentDegen} (in the following
  section) with $n=k$ and $r=\binom{\ell}{2}$ and $d=d(\ell)$, $H$ has
  an independent set with one vertex from each colour class. That is,
  in each set $\Q_{i,j}$ there is one path $Q_{i,j}$ such that
  $Q_{i,j}\cap Q_{a,b}=\emptyset$ for distinct pairs $i,j$ and
  $a,b$. Consider the set of paths
$$\PP:=\{P_i:i\in[\ell]\}\cup\{Q_{i,j}:1\leq i<j\leq\ell\}.$$ 
The intersection graph of \PP\ is bipartite and contains the
1-subdivision of $K_\ell$, which contains a $K_\ell$-minor. Therefore \PP\ is a
grid-like-minor of order $\ell$ in $G$.
\end{proof}

The next lemma with $r=2$ implies that if a graph $G$ contains a
grid-like-minor of order $\ell$, then the treewidth of $G$ is at least
$\ceil{\frac{\ell}{2}}-1$, which is the second part of
\thmref{GridLikeMinor}.

\begin{lemma}
  Let $H$ be the intersection graph of a set \X\ of connected
  subgraphs in a graph $G$. If $H$ contains a $K_\ell$-minor, and $H$
  contains no $K_{r+1}$-subgraph, then the treewidth of $G$ is at
  least $\ceil{\frac{\ell}{r}}-1$.
\end{lemma}

\begin{proof}
  Let $H_1,\dots,H_\ell$ be the branch sets of a $K_\ell$-minor in
  $H$. Each $H_i$ corresponds to a subset $\X_i\subseteq\X$, such that
  $\X_i\cap\X_j=\emptyset$ for distinct $i,j\in[\ell]$. Let $G_i$ be
  the subgraph of $G$ formed by the union of the subgraphs in
  $\X_i$. Since $H_i$ is connected and each subgraph in $\X_i$ is
  connected, $G_i$ is connected. For distinct $i,j\in[\ell]$, some
  vertex in $H_i$ is adjacent to some vertex in $H_j$. That is, some
  subgraph in $\X_i$ intersects some subgraph in $\X_j$. Hence $G_i$
  and $G_j$ share a vertex in common, and $\B:=\{G_1,\dots,G_\ell\}$
  is a bramble in $G$. Since $H$ has no $K_{r+1}$-subgraph, every
  vertex of $G$ is in at most $r$ bramble elements of \B. Thus every
  hitting set of \B\ has at least $\ceil{\frac{\ell}{r}}$
  vertices. Hence \B\ has order at least $\ceil{\frac{\ell}{r}}$. By
  \thmref{TreewidthBramble}, $G$ has treewidth at least
  $\ceil{\frac{\ell}{r}}-1$.
\end{proof}

\thmref{CartProd} follows from \thmref{GridLikeMinor} and the next
lemma.

\begin{lemma}
  \lemlabel{MinorCartProd} Let \PP\ be a grid-like-minor in a graph
  $G$.  Then the intersection graph $H$ of \PP\ is a minor of
  \CP{G}{K_2}.
\end{lemma}

\begin{proof}
  Let $\A\cup\B$ be a bipartition of $V(H)$. If $XY\in E(H)$ for some
  $X,Y\in\PP$, then $X\in\A$ and $Y\in\B$, and some vertex $v$ of $G$
  is in $X\cap Y$. Thus in \CP{G}{K_2}, the copy of $v$ in the first
  copy of $G$ is adjacent to the copy of $v$ in the second copy of
  $G$. Thus $H$ is obtained by contracting each path in $\A$ in the
  first copy of $G$, and by contracting each path in $\B$ in the
  second copy of $G$, as illustrated in \figref{Picture}.
\end{proof}

\begin{figure}[!h]
  \includegraphics[width=\textwidth]{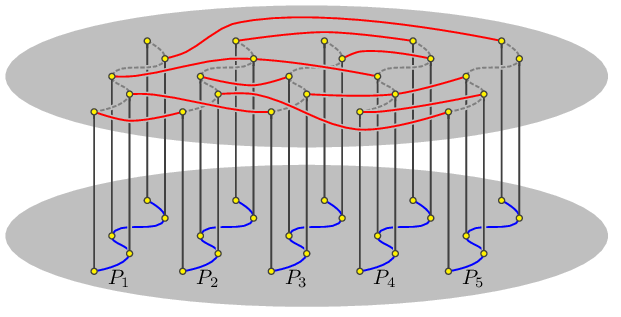}
  \vspace*{-5ex}
  \caption{\figlabel{Picture}Construction of a $K_{\ell}$-minor in
    \CP{G}{K_2}.}
\end{figure}

Note that \lemref{MinorCartProd} generalises as follows: If $H$ is the
intersection graph of a set of connected subgraphs of a graph $G$,
then $H$ is a minor of \CP{G}{K_{\chi(H)}}.

\section{Independent Transversals}

An \emph{independent transversal} in a coloured graph is an
independent set with exactly one vertex in each colour class.  Many
results are known that say that if each colour class is large compared
to the maximum degree and the number of colours, then an independent
transversal exists
\citep{King-JGT,LohSud-JCTB07,Alon-IJM88,Yuster-CPC97,Yuster-DM97,HS-CPCC06,BES-DM75,ST-Comb06}. Here
we prove two similar results, in which the maximum degree assumption
is relaxed. This result is used in the proof of
\thmref{GridLikeMinor}. The proof is based on the Lov\'{a}sz Local
Lemma.

\begin{lemma}[\citet{EL75}]
  \lemlabel{LLL} Let \X\ be a set of events, such that each event in
  \X\ has probability at most $p$ and is mutually independent of all
  but $D$ other events in \X. If $ep(D+1)\leq 1$ then with positive
  probability no event in \X\ occurs.
\end{lemma}

\begin{lemma}
  \lemlabel{ChooseIndependent} Let $V_1,\dots,V_r$ be the colour
  classes in an $r$-colouring of a graph $H$.  For $i\in[r]$, let
  $n_i:=|V_i|$, and let $m_i$ be the number of edges with one endpoint
  in $V_i$.  Suppose that $n_i\geq{2et}$ and $m_i\leq tn_i$ for some
  $t>0$ and for all $i\in[r]$.  Then there exists an independent set
  $\{x_1,\dots,x_r\}$ of $H$ such that each $x_i\in V_i$.
\end{lemma}

\begin{proof}
  Let $n:=\ceil{2et}$. Suppose that $n_i>n$ for some $i\in[r]$. Some
  vertex $v\in V_i$ has degree at least $\frac{m_i}{n_i}$. Thus
  $\frac{m_i-\deg(v)}{n_i-1}\leq\frac{m_i}{n_i}\leq t$. Hence $H-v$
  satisfies the assumptions. By induction, $H-v$ contains the desired
  independent set. Now assume that $n_i=n$ for all $i\in[r]$.

  For each $i\in[r]$, independently and randomly choose one vertex
  $x_i\in V_i$.  Each vertex in $V_i$ is chosen with probability
  $\frac{1}{n}$.  Consider an edge $vw$, where $v\in V_i$ and $w\in
  V_j$.  Let $X_{vw}$ be the event that both $v$ and $w$ are chosen.
  Thus $X_{vw}$ has probability $p:=\frac{1}{n^2}$.  Observe that
  $X_{vw}$ is mutually independent of every event $X_{xy}$ where
  $x\not\in V_i\cup V_j$ and $y\not\in V_i\cup V_j$.  Thus $X_{vw}$ is
  mutually independent of all but at most $D:=m_i+m_j-1$ other events.

  Now $2em_i\leq 2etn\leq n^2$ and $2em_j\leq 2etn\leq n^2$.  Thus
  $e(m_i+m_j)\leq n^2$.  That is, $ep(D+1)\leq 1$.  By \lemref{LLL},
  with positive probability no event $X_{vw}$ occurs.  Hence there
  exists $x_1,\dots,x_r$ such that no event $X_{vw}$ occurs.  That is,
  $\{x_1,\dots,x_r\}$ is the desired independent set.
\end{proof}

\begin{lemma}
  \lemlabel{ChooseIndependentDegen} Let $V_1,\dots,V_r$ be the colour
  classes in an $r$-colouring of a graph $H$. Suppose that $|V_i|\geq
  4e(r-1)d$ for all $i\in[r]$, and $H[V_i\cup V_j]$ is $d$-degenerate
  for distinct $i,j\in[r]$. Then there exists an independent set
  $\{x_1,\dots,x_r\}$ of $H$ such that each $x_i\in V_i$.
\end{lemma}

\begin{proof}
  Let $n:=\ceil{4e(r-1)d}$. For each $i\in[r]$, we may assume that
  $|V_i|=n$ (since deleting vertices from $V_i$ does not change the
  degeneracy assumption). Let $m_i$ be the number of edges with one
  endpoint in $V_i$. Every $d$-degenerate graph with $N$ vertices has
  at most $dN$ edges. Thus $m_i\leq 2(r-1)dn$. Let $t:=2(r-1)d$. The
  result follows from \lemref{ChooseIndependent} since $n\geq 2et$ and
  each $m_i\leq tn$.
\end{proof}

We now give an example that shows that the lower bound on $|V_i|$ in
\lemref{ChooseIndependentDegen} is best possible up to a constant
factor. Say $V_1$ has $d(r-1)$ vertices. Partition $V_1$ into sets
$W_2,\dots,W_r$ each of size $d$. Connect every vertex in $W_i$ to
every vertex in $V_i$ by an edge. Each bichromatic subgraph (ignoring
isolated vertices) is the complete bipartite graph $K_{d,n}$ (for some
$n$), which is $d$-degenerate. However, since every vertex in $V_1$
dominates some colour class, no independent set has one vertex from
each colour class. It is interesting to determine the best possible
lower bound on the size of each colour class in
\lemref{ChooseIndependentDegen}. It is possible that $|V_i|\geq
d(r-1)+c$ suffices.





\def\cprime{$'$} \def\soft#1{\leavevmode\setbox0=\hbox{h}\dimen7=\ht0\advance
  \dimen7 by-1ex\relax\if t#1\relax\rlap{\raise.6\dimen7
  \hbox{\kern.3ex\char'47}}#1\relax\else\if T#1\relax
  \rlap{\raise.5\dimen7\hbox{\kern1.3ex\char'47}}#1\relax \else\if
  d#1\relax\rlap{\raise.5\dimen7\hbox{\kern.9ex \char'47}}#1\relax\else\if
  D#1\relax\rlap{\raise.5\dimen7 \hbox{\kern1.4ex\char'47}}#1\relax\else\if
  l#1\relax \rlap{\raise.5\dimen7\hbox{\kern.4ex\char'47}}#1\relax \else\if
  L#1\relax\rlap{\raise.5\dimen7\hbox{\kern.7ex
  \char'47}}#1\relax\else\message{accent \string\soft \space #1 not
  defined!}#1\relax\fi\fi\fi\fi\fi\fi}

\end{document}